\documentclass[11pt]{amsart}

\usepackage{comment}
\usepackage[normalem]{ulem}
\usepackage[colorlinks,linkcolor=red,anchorcolor=blue,citecolor=blue]{hyperref} 
\usepackage{amscd} 
\usepackage{amsmath}
\usepackage{mathtools}
\usepackage{amssymb} 
\usepackage{amsthm}
\usepackage{dsfont}
\usepackage{bigdelim}
\usepackage{color} 
\usepackage{enumerate}
\usepackage{graphicx}
\usepackage{mathrsfs}
\usepackage{multirow}
\usepackage[all]{xy} 
\usepackage{color} 
\usepackage{enumitem}
\usepackage[dvipsnames]{xcolor}

\newtheorem{thm}{Theorem}[section] 
\newtheorem{theorem}[thm]{Theorem}
\newtheorem{corollary}[thm]{Corollary}
\newtheorem{proposition}[thm]{Proposition}

\newtheorem{lemma}[thm]{Lemma}
\theoremstyle{definition} 
\newtheorem{definition}[thm]{Definition}
\newtheorem{example}[thm]{Example}

\theoremstyle{remark}
\newtheorem{remark}[thm]{Remark}

\newtheorem{notation}[thm]{Notation}

\numberwithin{equation}{section}

\DeclareMathOperator{\Hom}{Hom}
\DeclareMathOperator{\Aut}{Aut}

\DeclareMathOperator{\GL}{GL}

\DeclareMathOperator{\vol}{vol}

\usepackage{dsfont}
\DeclareMathOperator{\R}{\mathds{R}}
\DeclareMathOperator{\C}{\mathds{C}}
\DeclareMathOperator{\CP}{\mathds{P}}
\DeclareMathOperator{\Q}{\mathds{Q}}
\DeclareMathOperator{\Z}{\mathds{Z}}

\usepackage{graphicx}

\makeatletter
\newcommand*\bigcdot{\mathpalette\bigcdot@{.5}}
\newcommand*\bigcdot@[2]{\mathbin{\vcenter{\hbox{\scalebox{#2}{$\m@th#1\bullet$}}}}}
\makeatother

\title[Two Non-Vanishing results \linebreak concerning the Anti-Canonical Bundle]{Two Non-Vanishing results \linebreak concerning the Anti-Canonical Bundle}

\author{Niklas M\"uller}
\address{Department of Mathematics, Universit\"at Duisburg-Essen,
Thea-Leymann-Str. 9, 45127 Essen, Germany.}
\email{{\tt niklas.mueller@uni-duisburg-essen.de}}


\keywords{While working on this project, the author was supported by the DFG Research Training Group 2553 Symmetries and Classifying Spaces: Analytic, Arithmetic and Derived}

\date{\today}

\begin{document}

\begin{abstract}

    Let $(X, \Delta)$ be a klt threefold pair with nef anti-log canonical divisor $-(K_X+\Delta)$. We show that $\kappa(X, -(K_X+\Delta))\geq 0$. To do so, we prove a more general equivariant non-vanishing result for anti-log canonical bundles, which is valid in any dimension.
    
\end{abstract}

\maketitle

\renewcommand{\thethm}{\Alph{thm}}
    \setcounter{tocdepth}{1}
    \tableofcontents

    \section{Introduction}

    \noindent
    Let $(X, \Delta)$ be a klt pair with nef log-canonical divisor $K_X + \Delta$. Recall that the celebrated \emph{Non-Vanishing conjecture} predicts that $\kappa(X, K_X + \Delta) \geq 0$, i.e.\ that there exists an integer $m\geq 1$ such that $m(K_X + \Delta)$ is Cartier and
       \begin{align*}
            H^0\Big(X, \mathcal{O}_X\big(m(K_X + \Delta) \big)\Big) \neq 0.
      \end{align*}
    Along with the \emph{Abundance conjecture}, which predicts that $K_X + \Delta$ is even semiample, the Non-Vanishing conjecture has attracted much attention.
    One might be tempted to ponder what happens on the opposite end of the spectrum, namely when the anti-log canonical divisor $-(K_X + \Delta)$ is nef. One quickly realises that the situation is somewhat different: Already the example of $\CP^2$ blown-up in $9$ points shows that $-K_X$ might be nef but fail to be semiample, see \cite{koike_UedaTheoryCounterExample}. On the other hand, in \cite{PeternellBauer_NefReductionAndACBundles} the authors classified rather explicitly smooth projective threefolds with nef anti-canonical class and observed a posteriori that $\kappa(X, -K_X)\geq 0$, at least when $X$ is rationally connected, see also \cite{xie_RCThreefoldsNefAcBundle} and \cite{xie_RCThreefoldsNefAcBundle2}. 

    Recently, Lazi\'c-Matsumura-Peternell-Tsakanikas-Xie \cite{lazic_NumericalNonVanishing} studied varieties with nef anticanonical bundle with a view towards the recently proposed generalised Non-Vanishing conjecture \cite{lazicPeternell_GeneralisedNonVanishing}, \cite{han_NumericalNonvanishingSurfacPairs}. In particular, they proved that $\kappa(X, -(K_X+\Delta))\geq 0$ if $(X, \Delta)$ is a rationally connected threefold pair with nef anti-log canonical bundle and they asked whether the same conclusion holds more generally, assuming only that $-(K_X + \Delta)$ is nef. In this paper we want to give the following affirmative answer to this question:
    
    \begin{theorem}
        Let $(X, \Delta)$ be a projective klt pair and assume that the anti-log canonical divisor $-(K_X + \Delta)$ is nef. Let $(F, \Delta_F)$ denote a general fibre of the MRC-fibration of $X$. If $-(K_F + \Delta_F)$ is semiample, then
        \begin{align*}
            \kappa(X,-(K_X + \Delta)) \geq 0.
        \end{align*}
        \label{Intro:Anti-Non-Vanishing}
    \end{theorem} 
    Despite giving only a partial answer, Theorem \ref{Intro:Anti-Non-Vanishing} applies in many cases. For example, as can be seen by using \cite{hausen_CoxRingVarietieswithTorusAction}, it applies whenever $F$ is a surface with enough automorphisms. Thus, combining our result with \cite{han_NumericalNonvanishingSurfacPairs} and \cite{lazic_NumericalNonVanishing} to deal with the other cases enables us to fully settle the three-dimensional case:
    \begin{corollary}
    \label{IntroCorollary}
        Let $(X, \Delta)$ be a projective klt threefold pair such that $-(K_X + \Delta)$ is nef. Then 
        \begin{align*}
            \kappa(X,-(K_X + \Delta)) \geq 0.
        \end{align*}
    \end{corollary}
   Besides the results obtained in \cite{cao_NefAnticanonicalBundleIII} and \cite{matsumuraWang_NefAnticanonicalBundle}, which form the backbone of our strategy, the main ingredient in our proof of Theorem \ref{Intro:Anti-Non-Vanishing} is the following `equivariant Non-Vanishing theorem', which allows us to lift sections from $F$ to $X$ as explained in Section \ref{sec:ReductionOfAtoB}. It holds in greater generality and might be of independent interest:
    \begin{theorem}\emph{(Equivariant Non-Vanishing for Anti-Canonical Divisors)}

        \noindent
        Let $(X, D)$ be a projective sub-log canonical pair and assume that $-(K_X + D)$ is semiample. Then for any commutative, linear algebraic subgroup $H\subseteq \Aut(X, D)$ there exists an integer $m\geq 1$ such that $-m(K_X + D)$ is Cartier and
        \begin{align*}
            H^0\Big(X, \mathcal{O}_X\big({-}m(K_X + D)\big)\Big)^H 
            \neq 0.
        \end{align*}
        We will abbreviate this by writing $\kappa(X, -(K_X + D))^H\geq 0$.\label{Intro:Equivariant Non-Vanishing}
    \end{theorem}
    Here, $\Aut(X, D)\subseteq \Aut(X)$ denotes the subgroup of automorphisms $\varphi \in \Aut(X)$ leaving $D$ invariant as a subset of $X$. Then $\Aut(X, D)$ acts on $H^0(X, \mathcal{O}_X(-m(K_X + D))\big))$ in a natural way and we denote by $H^0(X, \mathcal{O}_X({-}m(K_X + D)))^H $ the space of sections which are fixed under this action. For our terminology concerning pairs we refer the reader to Subsection \ref{subsec:Conventions} below.

    We would like to emphasise that the conclusion of Theorem \ref{Intro:Equivariant Non-Vanishing} is astonishing to the author in several different ways: First, perhaps somewhat unexpectedly, the crucial assumption in Theorem \ref{Intro:Equivariant Non-Vanishing} is the log canonicity and not the semipositivity of $-(K_X + D)$, as the conclusion of Theorem \ref{Intro:Equivariant Non-Vanishing} can fail even if $(X, D)$ is log smooth and if $-(K_X + D)$ is ample as soon as $(X, D)$ is no longer sub-log canonical, see Example \ref{ex:LogCanonicityIsOptimal}. 

    Moreover, it is noteworthy that the action of $H$ on the space of sections of $-m(K_X+D)$ is usually far from being trivial and already in very simple examples the ring of invariant sections of $-(K_X + D)$ might be much smaller than the ring of all sections, cf.\ Example \ref{ex:InvariantKodDimSmallerThanKodDim}. Nevertheless, Theorem \ref{Intro:Equivariant Non-Vanishing} shows that there always do exist invariant sections. In particular, in case $\kappa(X, -(K_X + D)) = 0$ the action \emph{must} be trivial after all. 

    Finally, note that the conclusion of Theorem \ref{Intro:Equivariant Non-Vanishing} certainly fails if one allows for more general subgroups of $\Aut(X, D)$, cf.\ Example \ref{ex:EquivKodDim-Infty}. 

    \subsection{Overview of the methods used in the proof of Theorem C}

    By far the most important special case of Theorem \ref{Intro:Equivariant Non-Vanishing} is the one where $(X, D)$ is log smooth and $H \cong (\C^\times)^r$ is an algebraic torus. Indeed, the general case is readily reduced to this one; the precise argument is contained in Section \ref{sec:Proof Of Main Results}.

    Under these assumptions the main idea of the proof is to convert the global question of whether $\kappa(X, -(K_X + D))^H\geq 0$ into a problem in convex geometry determined by the local action of $H$ near the fixed points in $X$. This method of computing global invariants of an action is well-known and heavily used in Geometric Invariant Theory and Symplectic Reduction. However, it is classically only developed for ample line bundles, while we will also need it for semiample ones. This case seems to have received little - if any - interest so far. Consequently, we re-prove the statements we need in detail. 

    In the log Fano case, the result boils down to a combination of some explicit local computations, some classical results about moment polytopes and a little piece of convex geometry. The semiample case then follows via some perturbation techniques. All of this will be explained in detail in Section \ref{sec:TorusInvariants}.

    \subsection{Conventions}
    \label{subsec:Conventions}

    Throughout the paper we work over the field $\C$ of complex numbers. 
    
    Since the terminology concerning pairs is slightly ambiguous at times in the literature, let us fix explicitly the notation we will use: For us, a \emph{pair} $(X, D)$ consists of a normal variety $X$ and a $\Q$-Weil divisor $D = \sum a_i D_i$ on $X$ with the property that $K_X + D$ is $\Q$-Cartier. Here, $D$ may or may not be effective.
    
    We call $(X, D)$ \emph{log smooth} if $X$ is smooth and if $D$ has SNC support. A log smooth pair $(X, D)$ will be called \emph{sub-log canonical} if $a_i \leq 1$ for all $i$, and \emph{log canonical} if additionally $D$ is required to be effective, i.e.\ $0\leq a_i \leq 1$. In general, a not necessarily log smooth pair $(X, D)$ is called (sub-)log canonical if for some/any log resolution $f\colon \hat{X} \rightarrow X$ the log smooth pair $(\hat{X}, \hat{D})$ is so, where $\hat{D}$ is determined by the rules $K_{\hat{X}} + \hat{D} \sim_{\Q} f^*(K_X + D)$ and $f_*\hat{D} = D$. Similarly, in the definition of a \emph{sub-klt} pair we do not require $D$ to be effective, while in the definition of a \emph{klt} pair we do. In other words, our definition of a sub-log canonical pair coincides precisely with the definition of a log canonical pair in \cite{kollar_BirationalGeometry}; similarly for (sub-)klt pairs.

    

    \renewcommand{\thethm}{\arabic{section}.\arabic{thm}}


    \section{Reduction of Theorem A to the Equivariant Non-Vanishing}
    \label{sec:ReductionOfAtoB}
	
	\noindent
    In this section we reduce the proof of Theorem \ref{Intro:Anti-Non-Vanishing} to the equivariant non-vanishing problem Theorem \ref{Intro:Equivariant Non-Vanishing}. The main technical tool we will use is the following structure theorem which was recently obtained in \cite{matsumuraWang_NefAnticanonicalBundle}, building on the previous works \cite{paun_FGNefACBundle}, \cite{paun_FGNefACBundleII}, \cite{zhang_VarietiesNefACBundle}, \cite{lu_SemiStabilityOfAlbanese}, \cite{cao_PhDThesis}, \cite{cao_NefAnticanonicalBundle}, \cite{cao_NefAnticanonicalBundleII}, \cite{cao_NefAnticanonicalBundleIII}, \cite{Ejiri_NefAntiCanonicalDivisorAndMRCFibration}, \cite{CCM_VarietiesNefACBundle} and \cite{wang_manifoldsWithNefAnticanonicalBundle} among others. We refer to \cite{Debarre_HigherDimensionalAlgebraicGeometry} for the definition of \emph{MRC-fibrations} (also known as \emph{rational quotients}). 

    \begin{theorem}\emph{(Matsumura-Wang, \cite{matsumuraWang_NefAnticanonicalBundle})}

        \noindent
        Let $(X, \Delta)$ be a projective klt pair with nef anti-log canonical divisor $-(K_X+\Delta)$. Then there exists a finite quasi-\'etale cover $\pi\colon X'\rightarrow X$ such that $X'$ admis a holomorphic, i.e.\ everywhere defined, MRC-fibration $f\colon X'\rightarrow Y$. Moreover, the following hold true:
        \begin{itemize}
            \item[(1)] Every component of $\Delta' := \pi^*\Delta$ is dominant over $Y$, i.e.\ $\Delta'$ is `\emph{horizontal}'. 
            \item[(2)] The pair $(Y, 0)$ is klt and $K_Y \sim_\mathds{Q} 0$.
            \item[(3)] The general fibre $(F, \Delta_F := \Delta'|_F)$ is a (connected) rationally connected klt pair with nef anti-log canonical divisor $-(K_F+\Delta_F)$.
            \item[(4)] The morphism $f$ is a \emph{locally constant fibration}, i.e.\ there exists a group homomorphism $\rho\colon \pi_1(Y) \rightarrow \Aut(F, \Delta_F)$ and an isomorphism over $Y$
            \begin{align*}
                (X', \Delta') \cong \Big( \widetilde{Y} \times (F, \Delta_F) \Big) / \pi_1(Y),
            \end{align*}
            where $\widetilde{Y}$ denotes the universal cover of $Y$ and where $\pi_1(Y)$ acts diagonally in the natural way on $\widetilde{Y}$ and through $\rho$ on $(F, \Delta_F)$.
        \end{itemize}\label{StructureVarietiesNefACBundle}
    \end{theorem}
    In view of Theorem \ref{StructureVarietiesNefACBundle} we expect $-(K_X+\Delta)$ to have sections if and only if $-(K_F+\Delta_F)$ admits sections which are invariant under the action of $\rho$. This is made precise below:
    \begin{theorem}
		Let $(F, \Delta_F)$ be a rationally connected projective klt pair and assume that the anti-log canonical divisor $-(K_F + \Delta_F)$ is nef. Then the following assertions are equivalent:
		\begin{itemize}
			\item[(1)] For any projective klt pair $(X, \Delta)$ with nef anti-log canonical divisor $-(K_X + \Delta)$ and whose MRC-fibration has $(F, \Delta_F)$ as its general fibre it holds that 
            \begin{align*}
                \kappa\big(X, -(K_X + \Delta)\big)\geq 0.
            \end{align*}
			\item[(2)] There exists a maximal algebraic torus $T\subseteq \Aut(F, \Delta_F)$ such that
                \begin{align*}
                    \kappa\Big(F, -\big(K_F + \Delta_F\big)\Big)^T\geq 0.
                \end{align*}
		\end{itemize}\label{criterionNonVanishing}
	\end{theorem}
    Recall that an \emph{algebraic torus} is an algebraic group which is isomorphic to $(\C^{\times})^r$ for some integer $r$. In order to prove Theorem \ref{criterionNonVanishing} we need the following elementary result; we include a proof for the sake of completeness:
    \begin{lemma}
        Let $G$ be a linear algebraic group acting algebraically on a finite dimensional vector space $V$. Then the following assertions are equivalent:
        \begin{itemize}
            \item[(1)] For any commutative, algebraic subgroup $H\subseteq G$ it holds that $V^H\neq 0$.
            \item[(2)] For any algebraic torus $T\subseteq G$ it holds that $V^T\neq 0$.
            \item[(3)] For some maximal algebraic torus $T\subseteq G$ it holds that $V^T\neq 0$.
        \end{itemize}\label{InvariantsCommutativeGroups}
    \end{lemma}
    \begin{proof}
        Clearly $(1)\Rightarrow(2)$. Regarding the converse, let $H\subseteq G$ be any commutative, algebraic subgroup. Then $H$ splits as
        \begin{align*}
            H = T \times U,
        \end{align*}
        where $T$ is an algebraic torus and $U$ is unipotent, cf.\ \cite[Theorem 16.13]{milne_algebraicGroups}. Consequently,
        \begin{align*}
            V^H = V^{(T\times U)} = \left(V^T\right)^U.
        \end{align*}
        Here, we used in the last step that the elements of $U, T$ commute so that the action of $U$ preserves $V^T$. Now, by our assumption $V^T\neq 0$. Since (essentially by definition) any algebraic action of a unipotent group on a non-trivial vector space fixes at least one non-trivial vector we conclude that $V^H\neq 0$.

        Finally, the equivalence of $(2)$ and $(3)$ is clear as any torus $T\subseteq G$ is contained in a maximal one and since any two maximal tori in $G$ are conjugate to each other, see \cite[Theorem 17.10]{milne_algebraicGroups}.
    \end{proof}

	\begin{proof}[Proof (of Theorem \ref{criterionNonVanishing})]
        First, let us prove that $(2)\Rightarrow (1)$. Fix $(X, \Delta)$ a projective klt pair which has $(F, \Delta_F)$ as general fibre of its MRC-fibration. We want to show that  $\kappa(X, -(K_X + \Delta))\geq 0$. Note that to do so we, may replace $(X, \Delta)$ by arbitrary quasi-\'etale covers, c.f.\ \cite[Theorem 5.13]{Ueno_ClassificationOfALgebraicVarieties}.  In particular, by Theorem \ref{StructureVarietiesNefACBundle}, we may assume that the MRC-fibration $f\colon X \rightarrow Y$ is a locally constant fibration with fibre $(F, \Delta_F)$ and such that $K_Y = 0$. Fix a group homomorphism $\rho \colon \pi_1(Y)\rightarrow \Aut(F, \Delta_F)$ such that 
        \begin{align*}
            (X, \Delta) = \Big(\widetilde{Y}\times \big(F, \Delta_F\big)\Big)/\pi_1(Y).
        \end{align*}
             \emph{Claim:} We may assume that the Zariski closure $H := \overline{\mathrm{Im}\rho} \subseteq \Aut(F)$ is a connected, commutative, linear algebraic group.\vspace{0.5\baselineskip}

        Indeed, as $X$ is a projective, the image of $\pi_1(Y) \overset{\rho}{\rightarrow} \Aut(F) \rightarrow \Aut(F)/\Aut^0(F)$ is finite, see for example \cite[Lemma 3.4]{mueller_LocallyConstantFibrations}. Hence, $H$ has only finitely many connected components. Replacing 
        $Y$ by the finite (!) \'etale cover corresponding to 
        \begin{align*}
            \ker\Big(\pi_1(Y)\rightarrow H/H^0\Big) \subseteq \pi_1(Y)
        \end{align*}
        and replacing $X$ by $X':= X\times_Y Y'$, we may assume that $H = H^0$ is connected. Moreover, as $F$ is rationally connected, $\mathrm{Pic}^0(F) = 0$ and so $G:= \Aut^0(F)$ is linear algebraic, see \cite[Corollary 2.18]{brion_AutomorphismGroups}. Thus, $H\subseteq G$ is also linear.
        
        Finally, according to \cite[Theorem B]{grebGuenanicaKebekus_KltCalabiYaus}, there exists a finite quasi-\'etale cover $Y'\rightarrow Y$ such that $Y' = A \times Z$ splits as a product of an abelian variety $A$ and a variety $Z$ with $\mathcal{O}_Z(K_Z) = \mathcal{O}_Z$ and vanishing augmented irregularity. Since $G = \Aut^0(F)$ is linear algebraic, the image of the map $\rho\colon \pi_1(Z) \rightarrow G = \Aut^0(F)$ is finite by \cite[Remark 1.4]{grebGuenanicaKebekus_KltCalabiYaus}. In other words, after another finite \'etale cover we may assume that $X$ splits as
        \begin{align*}
            X \cong \left((\widetilde{A} \times F)/\pi_1(A)\right) \times Z =: X'\times Z.
        \end{align*}
        As $\mathcal{O}_Z(K_Z) = \mathcal{O}_Z$ has sections, it suffices to prove the assertion for $(X', \Delta|_{X'})$ and we are thus reduced to the case that $Y = A$ is an abelian variety. In particular, in this case $\pi_1(Y)$ is an abelian group.

        Let us denote by $H\subseteq G$ the Zariski-closure of $\mathrm{Im}\rho$. Then, by continuity,
        \begin{align*}
            [ H, H ] = \Big[ \hspace{0.1cm} \overline{\mathrm{Im}\rho}, \overline{\mathrm{Im}\rho} \hspace{0.1cm} \Big] \subseteq \overline{ [ \mathrm{Im}\rho, \mathrm{Im}\rho ] } = \{ 1 \},
        \end{align*}
        and so $H\subseteq G$ is also commutative. In summary, $H$ is a connected, commutative, linear algebraic group. This concludes the proof of the \emph{Claim}.

        Let us now continue with the proof of $(2)\Rightarrow(1)$. According to Lemma \ref{InvariantsCommutativeGroups} and the \emph{Claim} we may find an integer $m\geq 1$ such that $-m\left(K_F + \Delta_F \right)$ is Cartier and such that
        \begin{align}
            \bigg( H^0 \Big(F, \mathcal{O}_F \Big( -m \big( K_F + \Delta_F \big) \Big) \Big) \bigg)^H\neq 0.\label{eqs:CritNV}
        \end{align}
        Here, we used that $G = \Aut^0(F)$ is linear algebraic. In what follows, we will prove that the elements of the vector space in (\ref{eqs:CritNV}) lift to non-zero sections in $H^0\left(X, \mathcal{O}_X\left(-m\left(K_X + \Delta\right)\right)\right)$. Indeed, note that
        \begin{align}
		H^0\Big(X, \mathcal{O}_X\big({-}m\left(K_X + \Delta\right)\big)\Big) = H^0\Big(Y, f_*\mathcal{O}_X\big({-}m\left(K_X + \Delta\right)\big)\Big).\label{eqs:CriterionNonVanishing1}
		\end{align}
        Now,
        \begin{align}
            f_*\mathcal{O}_X\big({-}m(K_{X} + \Delta)\big) & = f_*\mathcal{O}_X\big({-}m(K_{X/Y} + \Delta)\big) \nonumber \\
            & = \bigg(\widetilde{Y}\times H^0\Big(F, \mathcal{O}_F\big(-m(K_{F} + \Delta_F)\big)\Big)\bigg)/\pi_1(Y) \label{eqs:CriterionNonVanishing1b}
        \end{align}
        is the holomorphically flat vector bundle with fibre $\mathds{H} := \mathrm{H}^0(F, \mathcal{O}_F(-m(K_{F} + \Delta_F)))$ and monodromy representation $\pi_1(Y)\overset{\rho}{\rightarrow}\Aut(F, \Delta_F) \rightarrow \GL(\mathds{H})$, c.p.\ \cite[Proposition 6.3.(b)]{lazic_NumericalNonVanishing}. In particular,
        \begin{align*}
            H^0\Big(X, \mathcal{O}_X\big({-}m\big(K_X + \Delta\big)\big)\Big) 
            \overset{\textmd{(\ref{eqs:CriterionNonVanishing1})}}{=\joinrel=} \hspace{0.2cm} 
            & H^0\left(Y, f_*\mathcal{O}_X\left({-}m\left(K_{X/Y} + \Delta\right)\right)\right) \\
            \supseteq  \Big( & H^0\big(F, -m\left(K_F + \Delta_F \right) \big) \Big)^\rho \\
            \supseteq  \Big( & H^0\big(F, -m\left(K_F + \Delta_F \right) \big) \Big)^H \neq 0.
        \end{align*}
        Here we used in the last line that $\mathrm{Im}\rho\subseteq H$ and (\ref{eqs:CritNV}). Thus, $(2)\Rightarrow (1)$ is settled.

		Let us now turn to $(1)\Rightarrow (2)$; Given $(F, \Delta_F)$  as in the theorem and $T\subseteq\Aut(F, \Delta_F)$ any maximal algebraic torus, our goal is to produce a projective klt pair $(X, \Delta)$ which admits a locally constant fibration $f\colon X \rightarrow Y$ onto a variety $Y$ with $K_Y = 0$ and with fibre $(F, \Delta_F)$ such that
        \begin{align*}
            H^0\Big(X, \mathcal{O}_X\big({-}m\left(K_X + \Delta\right)\big)\Big) = H^0\Big(F, \mathcal{O}_F\big({-}m\big(K_F + \Delta_F \big) \big)\Big)^T.
        \end{align*}
        To this end, write $T \cong (\C^\times)^r$. Fix an abelian variety $Y$ of dimension $r$ and $\zeta \in \mathds{S}^1 \subseteq \mathds{C}^\times$ any number of absolute value one which is \emph{not} a root of unity. Consider the group homomorphism $\rho\colon \pi_1(Y)\rightarrow T$ determined by the rule
        \begin{align*}
		  \rho\colon \pi_1(Y)
        \rightarrow T, \qquad e_{2i-1} \mapsto (1, \dots, 1, \zeta, 1, \dots, 1), \quad e_{2i} \mapsto (1, \ldots, 1), \quad \forall i=1,\ldots, r,
		\end{align*}
        where $e_1, \ldots e_{2r}$ is some $\mathds{Z}$-basis for $\pi_1(Y)\cong \mathds{Z}^{2r}$.
        Then $\mathrm{Im}\rho$ is Euclidean dense in the compact group $(\mathds{S}^1)^r \subseteq (\C^\times)^r \cong T$. In particular, $\mathrm{Im}\rho$ is Zariski dense in $T$.

        Let us consider the complex analytic variety $X := (\widetilde{Y}\times F)/\pi_1(Y)$ equipped with the $\Q$-Weil divisor $\Delta := (\widetilde{Y}\times \Delta_F)/\pi_1(Y)$. Then $(X, \Delta)$ is a projective klt pair with nef anti-log canonical divisor $-(K_X+\Delta)$, see \cite[Theorem 5.1]{mueller_LocallyConstantFibrations}. By assumption there exists $m\geq 1$ such that
		\begin{align}
		0 \neq H^0\left(X, \mathcal{O}_X\big({-}m\left(K_X + \Delta\right)\big)\right) = H^0\left(Y, f_*\mathcal{O}_X\big({-}m\left(K_X + \Delta\right)\big)\right).\label{eqs:CriterionNonVanishing2}
		\end{align}
		Now, as in (\ref{eqs:CriterionNonVanishing1b}), 
        \begin{align*}
            f_*\mathcal{O}_X(-m(K_{X} + \Delta)) 
            = \bigg(\widetilde{Y}\times H^0\Big(F, \mathcal{O}_F\Big({-}m\big(K_{F} + \Delta_F\big)\Big)\Big)\bigg)/\pi_1(Y)
        \end{align*}
        is the holomorphically flat vector bundle with fibre $\mathds{H} := H^0(F, \mathcal{O}_F(-m(K_{F} + \Delta_F)))$ and monodromy representation $\pi_1(Y)\overset{\rho}{\rightarrow}\Aut(F, \Delta_F) \rightarrow \GL(\mathds{H})$. As $\mathrm{Im}\rho$ is contained in the compact subgroup $(\mathds{S}^1)^r\subsetneq T = (\mathds{C}^\times)^r$, this representation is unitary. It follows that all global sections of $f_*\mathcal{O}_X(-m(K_{X/Y} + \Delta))$ are flat, see for example \cite[Theorem 2.2.(b)]{wang_manifoldsWithNefAnticanonicalBundle}. In other words,
		\begin{align*}
			H^0\Big(Y, f_*\Big(\mathcal{O}_X\big({-}m(K_{X/Y} + \Delta)\big)\Big)\Big) = \mathds{H}^\rho = \mathds{H}^T.
		\end{align*}
		Here, in the last step we used that the action of $\Aut(F)$ on $\mathds{H} = H^0(F, -m(K_F + \Delta_F))$ is algebraic and that $\mathrm{Im}\rho\subsetneq T$ is Zariski dense by construction. As we know that $H^0(Y, f_*(-m(K_{X/Y} + \Delta)))\neq 0$ by (\ref{eqs:CriterionNonVanishing2}) we conclude.
	\end{proof}


 \section{Torus-Invariant Sections on Semiample Line Bundles}
 \label{sec:TorusInvariants}

    In this section we want to prove Theorem \ref{Intro:Equivariant Non-Vanishing} under the additional assumption that $(X, D)$ is log smooth:
    \begin{theorem}
        Let $(X, D)$ be a projective, log smooth, sub-log canonical pair such that $-(K_X+D)$ is semiample. Then for any algebraic torus $T\subseteq \Aut(X, D)$ it holds that
        \begin{align*}
            \kappa(X, -(K_X+D))^T \geq 0.
        \end{align*}\label{Equivariant Non-Vanishing: Semi ample case}
    \end{theorem}
    To prove Theorem \ref{Equivariant Non-Vanishing: Semi ample case} we will employ some discrete methods from toric geometry. Let us start by introducing the following concept:
    \begin{definition}
        Let $X$ be a normal variety and let $G\subseteq \Aut(X)$ be a subgroup. A coherent sheaf $\mathscr{L}$ on $X$ is said to be \emph{$G$-invariant} if $g^*\mathscr{L}\cong\mathscr{L}$ for all $g\in G$. In this case we say that $\mathscr{L}$ is \emph{$G$-linearisable} if the action of $G$ on $X$ may be lifted to an action of $G$ on the total space of $\mathscr{L}$ via sheaf automorphisms. A \emph{$G$-linearisation} of $\mathscr{L}$ is a choice of such a lift. See \cite[Definition 1.6]{Mumford_GIT} for a more formal definition of linearisations.
    \end{definition}
    \begin{example}
        Let $X$ be a normal variety, let $G\subseteq \Aut(X)$ be a subgroup and let $\mathscr{L}, \mathscr{L}_1, \mathscr{L}_2$ be $G$-linearised coherent sheaves on $X$. Then also $\mathscr{L}^{*}$ and $\mathscr{L}_1\otimes \mathscr{L}_2$ inherit natural $G$-linearisations.
    \end{example}
    \begin{definition}
        Let $X$ be a normal variety, let $G\subseteq \Aut(X)$ be a subgroup and let $\mathscr{L}$ be a $G$-linearised, coherent, reflexive $\mathcal{O}_X$-module of rank one. Then $G$ acts on $H^0(X, (\mathscr{L}^{\otimes m})^{**})$ for any integer $m$ and we will say that that $\kappa(X, \mathscr{L})^G \geq 0$ if and only if
        \begin{align*}
             \bigoplus_{m=1}^\infty H^0\Big(X, \left(\mathscr{L}^{\otimes m}\right)^{**}\Big)^G \neq 0.
        \end{align*}\label{def:EquivariantKodairaDimension}
    \end{definition}
    \begin{example}
        Consider $X= \CP^1$ equipped with the line bundle $\mathscr{L} = \mathcal{O}_{\CP^1}(-1)$.
        \begin{itemize}
            \item[(1)] Let $G = \Aut(\CP^1) = \textmd{PGL}_2(\mathds{C})$. Then $\mathscr{L}$ is clearly $G$-invariant. However, this action is (rather famously) \emph{not} linearisable, see for example \cite[Example 4.2.4]{brion_LinearisationsOfGroupActions}.
            \item[(2)] Let $G = T \subsetneq \Aut(\CP^1)$ be the subgroup of (equivalence classes of) diagonal matrices. Then $ \C^\times \cong T \subsetneq \Aut(\CP^1)$ is a maximal torus acting via $t\bigcdot [x:y] := [tx:y]$ and this action \emph{is} linearisable. In fact, for any $w\in \mathds{Z}$ an explicit choice of lifting is provided by $t\bigcdot (x, y) = (t^{w+1}x, t^w y)$. Here we identify, as per usual, $ \mathcal{O}_{\CP^1}(-1)|_{[x:y]} = \mathds{C}\cdot (x, y) \subseteq \mathds{C}^2$.
        \end{itemize}\label{ex:LinearisationsOnP1}
    \end{example}

    \begin{example}
        Let $X$ be a normal variety and let $G\subseteq \Aut(X)$ be a subgroup. Then the anti-canonical bundle $\mathcal{O}_X(-K_X)$ admits a natural linearisation given by push-forward of forms: $g\bigcdot \tau := g_*(\tau)$. Recall that in local coordinates $\varphi_*$ is simply defined by the formula
            \begin{align*}
                \left(\varphi_*\left(f \frac{\partial}{\partial z_1}\wedge\ldots\wedge\frac{\partial}{\partial z_n}\right)\right)(\varphi(x)) 
                = f(x) \cdot d\varphi\left(\frac{\partial}{\partial z_1}\right)\wedge\ldots\wedge d\varphi \left(\frac{\partial}{\partial z_n}\right),
            \end{align*}
            where $\varphi\colon X\rightarrow X$ is any automorphism.
            \label{ex:LinearisationsOnTheAC}
    \end{example}
    \begin{example}
        Let $D$ be a Cartier divisor on $X$ and assume that it is invariant (as a subset of $X$) under the action of $G$. Then also $\mathcal{O}_X(D)$ carries a natural $G$-linearisation as a subsheaf of $\C(X)$, which is induced by the pull-back of functions: $g\bigcdot f := (g^{-1})^*f$. Note that we take the inverse in order to obtain a left-action.

        Similarly, $\Omega^1_X$ admits a natural linearisation given by $g\bigcdot df = d(g\bigcdot f) = d((g^{-1})^*f)$.
        \label{ex:LinearisationsOnSheavesOfDivisors}
    \end{example}
    In conclusion, whenever $(X, D)$ is a pair, then the line bundle $\mathcal{O}_X(-m(K_X + D))$ carries a natural $\Aut(X, D)$-linearisation for all $m\in \Z$ such that $m(K_X+D)$ is Cartier.\\
    
    In the following we will require some standard facts about algebraic tori and their representations which we recall below: Let $T \cong (\C^\times)^r$ be an algebraic torus. Then its \emph{character lattice} $M := \Hom(T, \C^\times)$ is a finitely generated free abelian group of rank $\textmd{rk} M = \dim T$. Explicitly, if $T = (\C^\times)^r$ then 
        \begin{align*}
            \Z^r \overset{\sim}{\longrightarrow} M= \Hom(T, \C^\times), \quad w = (a_1, \ldots, a_r) \mapsto \Big( t \mapsto t^{w} := t_1^{a_1} \cdot \ldots \cdot t_r^{a_r} \Big).
        \end{align*}
        In the sequel, following standard practice, we will consider $M$ as an additive group $(M, +)$. In particular, the neutral element $0\in M$ corresponds to the trivial homomorphism $t\mapsto 1$. We will also consider the real vector space $M_{\R} := M \otimes_{\Z} \R$.

        Now, let $\rho\colon T\rightarrow \GL(V)$ be an algebraic representation of $T$ on a finite-dimensional complex vector space $V$ (we also say that $V$ is a \emph{$T$-module}). Then the action of $T$ can be \emph{diagonalised}, i.e.\ there exists a $\C$-basis $e_1, \ldots, e_m$, called a basis of \emph{eigenvectors}, for $V$ such that 
        $$
            t \bigcdot e_i = t^{w_i} \cdot e_i, \quad \forall t\in T
        $$
        for some $w_1, \ldots, w_m \in M= \Hom(T, \C^\times)$ called the \emph{weights} of the representation $\rho$. A representation of $T$ is determined up to isomorphism by its weights. 
    
    \begin{definition}
        Let $X$ be a normal projective variety, let $T\subseteq \Aut(X)$ be an algebraic torus and let $\mathscr{L}$ be a coherent sheaf on $X$, linearised for the action of $T$. Then $T$ acts naturally on the space $H^0(X, \mathscr{L})$. We let $W_{\Gamma}(X, \mathscr{L}) \subseteq M$ denote the set of  weights of this action. Moreover, we denote the convex hull of $W_{\Gamma}(X, \mathscr{L})$ in $M_{\R}$ by $P_{\Gamma}(X, \mathscr{L})$ and call this the \emph{section polytope} of $\mathscr{L}$.
    \end{definition}

    \begin{notation}
        From now on and for the rest of this section we let $X$ denote a smooth projective variety and we fix an algebraic torus $T\subseteq \Aut(X)$. Then the set of \emph{fixed points} $X^T := \{ x\in X | \hspace{0.1cm} t\bigcdot x = x, \forall t\in T \}$ is a non-empty, closed, \emph{smooth} subvariety of $X$, see \cite[Theorem 13.1, Proposition 13.20]{milne_algebraicGroups}. We let $X^T = Y_1 \sqcup \ldots \sqcup Y_c$ denote the decomposition of $X^T$ into its connected components. Moreover, for any $i = 1,\ldots , c$ let us fix  a point $y_i \in Y_i$.
    \end{notation}
    \begin{definition}
        Let $\mathscr{L}$ be a line bundle on $X$, linearised for the action of $T$. Then $T$ acts on the one-dimensional $\C$-vector space $\mathscr{L}|_{y_i}$ for any $i = 1,\ldots , c$, with weight $\mu_i$ say. Let us set $W_{\mu}(X, \mathscr{L}) := \{ \mu_1, \ldots, \mu_c \} \subseteq M$. The convex hull $P_{\mu}(X, \mathscr{L})$ of $W_{\mu}(X, \mathscr{L})$ in  $M_{\mathds{R}}$ is called the \emph{moment polytope}.
    \end{definition}
    \begin{remark}
        Note that $\mu_i$ does not depend on the choice of $y_i \in Y_i$. Indeed, let $y_i\in U \subseteq X$ be a $T$-invariant affine open neighbourhood which exists by \cite[Corollary 5.3.6]{brion_LinearisationsOfGroupActions}. Pick a section $\sigma \in \mathscr{L}(U)$ which does not vanish at $y_i$. Then the weight $\mu_y$ of the action of $T$ on $\mathscr{L}|_{y}$ is given by $t^{\mu_y} =  \frac{(t\bigcdot\sigma)(y)}{\sigma(y)}$ for all $y\in Y_i\cap U$ such that $\sigma(y)\neq 0$. This shows that the map $Y_i \rightarrow M$,  $y\mapsto \mu_y$ is continuous, hence constant.
    \end{remark}
    \begin{example}
        Continuing  Example \ref{ex:LinearisationsOnP1}, let $\C^{\times}\subseteq \Aut(\CP^1)$ act on $\mathcal{O}_{\mathds{P}^1}(-1)$ via the rule $t\bigcdot (x,y) = (t^{w+1}x, t^wy)$. Then $W_{\mu}(\CP^1, \mathcal{O}(-1)) = \{ w, w+1\} \subseteq M = \Z$ and, consequently, $P_{\mu}(\CP^1, \mathcal{O}_{\mathds{P}^1}(-1)) = [  w+1, w ] \subseteq M_{\R} = \R$.\label{ex:LinearisationsOnP1_2}
    \end{example}
    The following result is essentially taken from \cite[Lemma 2.4]{wisniewski_AlgebraicTorusActionsOnContactManifolds} where it is stated only for ample bundles. The proof generalises to the semiample setting without difficulties.
    \begin{proposition}
        Let $X$ be a smooth projective variety, let $T\subseteq \Aut(X)$ be an algebraic torus and let $\mathscr{L}$ be a semiample line bundle on $X$, linearised for the action of $T$.
        
        Then for any integer $m\geq 1$ the following hold true:
        \begin{itemize}
            \item[(1)] $mP_{\Gamma}(X, \mathscr{L}) \subseteq P_{\Gamma}(X, \mathscr{L}^{\otimes m})$ and $mP_{\mu}(X, \mathscr{L}) = P_{\mu}(X, \mathscr{L}^{\otimes m})$ as subsets of $M_{\mathds{R}}$,
            \item[(2)] $P_{\Gamma}(X, \mathscr{L}) \subseteq P_{\mu}(X, \mathscr{L})$ and
            \item[(3)] if $\mathscr{L}$ is basepoint free then $P_{\Gamma}(X, \mathscr{L}) = P_{\mu}(X, \mathscr{L})$.
        \end{itemize}\label{PropertiesSectionMomentPolytope}
    \end{proposition}
    \begin{proof}
        $(1)$: Clearly $mW_{\Gamma}(X, \mathscr{L}) \subseteq W_{\Gamma}(X, \mathscr{L}^{\otimes m})$ for if $w_1, \ldots, w_m \in W_{\Gamma}(X, \mathscr{L})$ are weights for the action of $T$ on $H^0(X, \mathscr{L})$ with corresponding eigenvectors $\sigma_1, \ldots, \sigma_m$ then 
        \begin{align*}
            \sigma_1\otimes \ldots\otimes\sigma_m\in H^0\Big(X, \mathscr{L}^{\otimes m}\Big)
        \end{align*}
        is an eigenvector for the action of $T$ of weight $w_1 + \ldots + w_m \in mW_{\Gamma}(X, \mathscr{L})$. We deduce that $mP_{\Gamma}(X, \mathscr{L}) \subseteq P_{\Gamma}(X, \mathscr{L}^{\otimes m})$.

        That $mP_{\mu}(X, \mathscr{L}) = P_{\mu}(X, \mathscr{L}^{\otimes m})$ is obvious for if $y\in X^T$ and if $\mu$ is the weight for the action of $T$ on $\mathscr{L}|_y$ then the weight of the action on $\mathscr{L}^{\otimes m}|_y$ is just $m\mu$.

        To prove $(2)$, the argument in \cite[Lemma 2.4.(2)]{wisniewski_AlgebraicTorusActionsOnContactManifolds} applies ad verbatim. Note that to prove Theorem \ref{Equivariant Non-Vanishing: Semi ample case} we will not make use of this inclusion.

        Regarding $(3)$, the proof of  \cite[Lemma 2.4.(3)]{wisniewski_AlgebraicTorusActionsOnContactManifolds} again goes through without changes. However, as the argument is so elementary we want to quickly repeat it here: Fix $y\in X^T$ and let us denote by $\mu$ the weight of the action of $T$ on $\mathscr{L}|_y$. As $\mathscr{L}$ is generated by global sections we have the short exact sequence of $T$-modules
        \begin{align*}
            0 \rightarrow H^0\big(X, \mathscr{L}\otimes \mathfrak{m}_y\big) \rightarrow H^0\big(X, \mathscr{L}\big) \rightarrow \mathscr{L}|_y \rightarrow 0.
        \end{align*}
        As it is well-known that any short exact sequence of $T$-modules splits, we find an eigenvector $\sigma \in \mathrm{H}^0(X, \mathscr{L})$ for the action of $T$ of weight $\mu$. This proves that $W_{\mu}(X, \mathscr{L}) \subseteq W_{\Gamma}(X, \mathscr{L})$. We infer that $P_{\mu}(X, \mathscr{L}) \subseteq P_{\Gamma}(X, \mathscr{L})$ and so the result follows from $(2)$.
    \end{proof}

    \begin{corollary}
        Let $X$ be a smooth projective variety, let $T\subseteq \Aut(X)$ be an algebraic torus and let $\mathscr{L}$ be a semiample line bundle on $X$, linearised for the action of $T$.

        Then $\kappa(X, \mathscr{L})^T \geq 0$ if and only if the origin $0\in M$ is contained in $P_{\mu}(X, \mathscr{L})$.\label{PropertiesSectionMomentPolytope Part 2}
    \end{corollary}
    \begin{proof}
        Let us start by proving that 
        \begin{align}
            \kappa(X, \mathscr{L})^T \geq 0 
            \quad \Leftrightarrow \quad \exists m: \hspace{0.25cm} 0 \in W_{\Gamma}\big(X, \mathscr{L}^{\otimes m}\big)
            \quad \Leftrightarrow \quad \exists m: \hspace{0.25cm} 0 \in P_{\Gamma}\big(X, \mathscr{L}^{\otimes m}\big).\label{eqs:PropertiesSectionMomentPolytope}
        \end{align}
        Indeed, the first equivalence holds true by the very definition of $W_{\Gamma}$ and also the second `$\Rightarrow$'-implication is obvious. Regrading the converse, assume that $0\in P_{\Gamma}(X, \mathscr{L}^{\otimes m})$. Let us denote the weights of the action of $T$ on $\mathrm{H}^0(X, \mathscr{L}^{\otimes m})$ by $w_1, \ldots, w_\ell$ and let $\sigma_1, \ldots, \sigma_\ell$ be a corresponding basis of eigenvectors. Now $0\in P_{\Gamma}(X, \mathscr{L}^{\otimes m})$ simply means that there exist real numbers $0\leq \lambda_i \leq 1$ such that $\sum \lambda_i = 1$ and 
        \begin{align*}
            \sum_{i=1}^\ell \lambda_i w_i = 0.
        \end{align*}
        In fact, by Proposition \ref{realAndRationalConvexHull} below one can assume the $\lambda_i$’s to be rational numbers, say $\lambda_i = \frac{p_i}{q_i}$. Set $k_i := q_1 \cdot \ldots \cdot q_{i-1} \cdot p_i \cdot q_{i+1} \cdot \ldots \cdot q_\ell \in \Z$. Then
       \begin{align*}
            \sum_{i=1}^\ell k_i w_i = 0.
        \end{align*}
        But this in turn just means that 
        \begin{align*}
            \sigma_1^{\otimes k_1} \otimes \ldots \otimes \sigma_\ell^{\otimes k_\ell} \in H^0\Big(X, \mathscr{L}^{\otimes m(k_1+\ldots + k_\ell)}\Big)
        \end{align*}
        is an eigenvector for the action of $T$ of weight $\sum k_iw_i = 0$. In other words, denoting $k = k_1 +\ldots + k_\ell$ we deduce that
       \begin{align*}
            0 \in kW_{\Gamma}\Big(X, \mathscr{L}^{\otimes m}\Big) 
	       \subseteq kP_{\Gamma}\Big(X, \mathscr{L}^{\otimes m}\Big) 
	       \subseteq P_{\Gamma}\Big(X, \mathscr{L}^{\otimes km}\Big),
        \end{align*}
        where we applied Proposition \ref{PropertiesSectionMomentPolytope}. This concludes the proof of (\ref{eqs:PropertiesSectionMomentPolytope}). 
        
        But now, a further application of Proposition \ref{PropertiesSectionMomentPolytope} yields
        \begin{align*}
            P_{\Gamma}\Big(X, \mathscr{L}^{\otimes m}\Big) = P_{\mu}\Big(X, \mathscr{L}^{\otimes m}\Big) = mP_{\mu}\big(X, \mathscr{L}\big)
        \end{align*}
        for sufficiently large $m\geq 1$. In view of (\ref{eqs:PropertiesSectionMomentPolytope}) we deduce that $\kappa(X, \mathscr{L})^T \geq 0$ if and only if $0 \in P_{\mu}(X, \mathscr{L})$ as proclaimed.
    \end{proof}
    The following result is just basic linear algebra. Due to the inability of the author to locate a reference in the literature we include a proof anyway.
    \begin{proposition}
        Let $M$ be a finitely generated free abelian group, let $w_1, \ldots, w_m\in M_{\Q} := M\otimes \Q$ be a finite subset and let $P\subseteq  M_{\R}$ be its convex hull. Then
        \begin{align*}
            P \cap M_{\Q} = P_{\Q} := \left\{ \sum_{i=1}^m \lambda_i w_i \hspace{0.1cm} \Big| \hspace{0.1cm} 0\leq \lambda_i, \hspace{0.1cm} \lambda_i \in \Q, \hspace{0.1cm} \sum_{i=1}^m \lambda_i = 1  \right\}.
        \end{align*}\label{realAndRationalConvexHull}
    \end{proposition}
    \begin{proof}
        Clearly $P_{\mathds{Q}} \subseteq P \cap M_{\Q}$. Regarding the converse, fix $w\in P \cap M_{\Q}$. After replacing $w_i$ by $w_i - w$ we may assume that $w= 0$. Write
        \begin{align*}
            w = \sum_{i=1}^m \lambda^*_i w_i
        \end{align*}
        for some real numbers $0\leq \lambda^*_i$ such that $ \sum_{i=1}^m \lambda^*_i = 1$. Forgetting some of the $w_i$ if necessary, we may assume that $\lambda^*_i>0$ for all $i=1,\ldots, m$.

        Now, consider the rational system of linear equations
        $$
            \left\{ \left. \left(\begin{matrix}
                \lambda_1 \\ \vdots \\ \lambda_m 
            \end{matrix} \right) \hspace{0.1cm}
            \right| \hspace{0.1cm} \sum_{i=1}^m \lambda_i w_i = 0, \hspace{0.1cm} \sum_{i=1}^m \lambda_i = 1
            \right\}.
        $$
        By assumption it admits the real solution $\lambda^* := (\lambda_1^*, \ldots, \lambda_m^*)$ such that $\lambda_i^* > 0$. Let $(\lambda_1, \ldots, \lambda_m)$ be any rational solution which is sufficiently close to $\lambda^*$. Then also $\lambda_i >0$ and so $w = \sum_{i=1}^m \lambda_i w_i$ with rational $\lambda_i$ as required.
    \end{proof}
    \begin{example}
        Continuing  Example \ref{ex:LinearisationsOnP1}, let $\C^{\times}\subseteq \Aut(\CP^1)$ act on $\mathcal{O}(-1)$ via the rule $t\bigcdot (x,y) = (t^{w+1}x, t^wy)$. By Example \ref{ex:LinearisationsOnP1_2} the moment polytope for the induced action of $T$ on $\mathcal{O}(m)$ is given by
        \begin{align*}
            P_{\Gamma}\big(\CP^1, \mathcal{O}(m)\big) = P_{\mu}\big(\CP^1, \mathcal{O}(m)\big) = \Big[-m(w+1), -mw\Big].
        \end{align*}
        We deduce that $\kappa(\CP^1, \mathcal{O}(1))^T \geq 0$ if and only if $w=0$ or $w=-1$.
    \end{example}

    \begin{proposition}
        Let $X$ be a smooth projective variety and let $T\subseteq \Aut(X)$ be an algebraic torus. Choose a fixed point $y\in X^T$ and denote by $\nu_1, \ldots, \nu_n$ the weights of the natural isotropy action of $T$ on the cotangent space $\Omega_X^1|_{y}$.
        \begin{itemize}
            \item[(1)] The weight $\mu$ of the action of $T$ on $\mathcal{O}_X(-K_X)|_y$ satisfies the equation
            \begin{align*}
                \mu = -(\nu_1 + \ldots + \nu_n).
            \end{align*}
            \item[(2)] Let $D$ be a $T$-invariant prime divisor on $X$ and let us denote by $\mu$ the weight of the action of $T$ on $\mathcal{O}_X(-D)|_y$. Then $\mu = \sum_i m_i \cdot \nu_i$ for some non-negative integers $m_i \geq 0$. Moreover:
            \begin{itemize}
                \item[(a)] If $y\notin D$ then $\mu = 0$.
                \item[(b)] If $y\in D$ is a smooth point of $D$ then the natural isomorphism 
                $$\mathcal{O}_X(-D)|_y\cong \mathcal{N}^*_{D/X}|_{y} \subseteq \Omega_X^1|_{y}$$ 
                is compatible with the action of $T$. In particular, $\mu = \nu_i$ for some $i$.
            \end{itemize}
        \end{itemize}\label{ExampleComputationLocalWeights}
    \end{proposition}
    \begin{proof}
       The first statement is clear as $\mathcal{O}_X(-K_X)|_y = \det(\mathrm{T}X|_{y})$ as $T$-modules.

       Regarding item $(2)$, we claim that there exists a local defining equation $f \in \mathcal{O}_{X, y}$ of $D$ such that 
       \begin{align}
           t \bigcdot f = t^\mu \cdot f, \quad \forall t \in T.\label{eq:ExampleComputationLocalWeights}
       \end{align}
       Indeed, let $f' \in \mathcal{O}_{X, y}$ be any local defining equation. Since the action of $T$ on $\mathcal{O}_{X, y}$ is algebraic, the $\C$-vector space spanned by $\{ t \bigcdot f' | \ t \in T\}$ is finite-dimensional, see \cite[Corollary 4.8]{milne_algebraicGroups}. Let $f_1, \ldots, f_r$ be a basis of $T$-eigenvectors, of weights $\mu_i$ say. Since $t\bigcdot f' = (t^{-1})^*f'$ is a local defining equation of $D$ for any $t\in T$ by Example \ref{ex:LinearisationsOnSheavesOfDivisors}, it follows that $f_i$ vanishes along $D$ for any $i$. On the other hand, since $f'$ is a $\C$-linear combination of the $f_i$, there must exist at least one $f = f_{i^*}$ which has multiplicity precisely one along $D$. Then $f$ is a local defining equation for $D$, i.e.\ $f|_y = f_{i^*}|_y$ generates $\mathcal{O}_X(-D)|_y$. In particular, the weight of the action of $T$ on $\mathcal{O}_X(-D)|_y$ is $\mu_{i^*}$. Thus, $\mu = \mu_{i^*}$, thereby proving (\ref{eq:ExampleComputationLocalWeights}).
       
        Now, it follows from Luna's \'etale slice theorem, c.f.\ for example \cite[Lemma 13.36]{milne_algebraicGroups}, that
        \begin{align*}
            \widehat{\mathcal{O}_{X, y}} \cong \C [\![ z_1, \ldots, z_n ]\!] 
        \end{align*}
        for some formal functions $z_i$ satisfying $t \bigcdot z_i = t^{\nu_i} \cdot z_i$ for any $i = 1, \ldots n$. Express 
        \begin{align*}
            f = \sum_{J \in (\mathds{Z}_{\geq 0})^n} a_J \cdot z^J
        \end{align*}
        as a formal power series. Since $ t \bigcdot f = t^\mu \cdot f$ by (\ref{eq:ExampleComputationLocalWeights}), it follows that $a_J = 0$ for all but possibly those $J = (m_1, \ldots, m_n)$ such that $\sum m_i \cdot \nu_i = \mu$. As $f \neq 0$, we deduce that there exists at least one $J = (m_1, \ldots, m_n) \in (\mathds{Z}_{\geq 0})^n$ such that $\mu = \sum_i m_i \cdot \nu_i$. This proves the assertion.

        Finally, if $y\notin D$ we may take $f = 1$, the constant function. As $t\bigcdot 1 = (t^{-1})^*(1) = 1$ we conclude that $T$ acts trivially on $\mathcal{O}_X(-D)|_y$ as proclaimed.

        Moreover, if $y\in D$ is a smooth point of $D$ then $\mathcal{N}^*_{D/X}|_y \subseteq \Omega^1_{X}|_y$ is a one-dimensional subspace, generated by $df|_y$. We compute
        \begin{align*}
            t \bigcdot df|_y \overset{\textmd{Ex.\ \ref{ex:LinearisationsOnSheavesOfDivisors}}}{=\joinrel=\joinrel=\joinrel=} d(t \bigcdot f)|_y = d(t^\mu \cdot f)|_y = t^\mu df|_y.
       \end{align*}
       Thus the weights of the actions of $T$ on $\mathcal{O}_X(-D)|_y$ and $\mathcal{N}^*_{D/X}|_{y}$ are the same and we are done.
    \end{proof}
    
    The presence of the divisor $A$ in the following Lemma should be understood to be of a purely technical nature. Note that in any case Lemma \ref{EquivariantNonVanishing Ample Case} is no more general than Theorem \ref{Intro:Equivariant Non-Vanishing} as also the pair $(X, D - A)$ is sub-log canonical, see \cite[Corollary 2.35]{kollar_BirationalGeometry}.

    \begin{lemma}
        Let $(X, D)$ be a projective, log smooth, sub-log canonical pair and fix an algebraic torus $T\subseteq \Aut(X, D)$. Let $A \geq 0$ be an effective, $T$-invariant $\mathds{Q}$-divisor. 

        If $-(K_X+ D - A)$ is ample, then $0 \in P_{\mu}(X, -(K_X+ D  - A))$.\label{EquivariantNonVanishing Ample Case}
    \end{lemma}
    \begin{remark}
        Note that in view of item $(1)$ in Proposition \ref{PropertiesSectionMomentPolytope} it makes sense to define the moment polytope $P_{\mu}(X, -(K_X+ D  - A))$ even if $A, D$ are only $\Q$-divisors by choosing $m$ such that $-m(K_X+ D  - A)$ is Cartier and setting
        \begin{align*}
            P_{\mu}\Big(X, -\big(K_X+ D  - A\big)\Big) := \frac{1}{m} P_{\mu}\Big(X, \mathcal{O}_X\big(-m\big(K_X+ D  - A\big)\big)\Big) \subseteq M_{\mathds{R}}.
        \end{align*}
    \end{remark}

    \begin{proof}[Proof (of Lemma \ref{EquivariantNonVanishing Ample Case})]
    Choose $y\in X^T\neq \emptyset$ and let us denote by $\mu \in M_{\mathds{R}}$ the weight of the action of $T$ on $\mathcal{O}_X(-(K_X+ D  - A))|_y$ \footnote{Here, to be precise we should really choose $m\geq 1$ such that $-m(K_X+ D  - A)$ is Cartier, denote the weight of the action of $T$ on $\mathcal{O}_X(-m(K_X+ D  - A))|_y$ by $\mu'$ and set $\mu := \frac{\mu'}{m}$.}. Below, we will prove that 
    \begin{align}
        (1-\varepsilon) \mu \in P_{\mu}\Big(X, {-}\big(K_X+ D  - A\big)\Big)\label{eqs:EquivNonVan AmpleCase 1}
    \end{align}
    for all sufficiently small $\varepsilon \geq 0$. Using some completely elementary convex geometry this immediately yields the result, cf. Proposition \ref{Convex Geometry} below.

     To this end, let us denote the weights of the action of $T$ on $\Omega_X^1|_{y}$ by $\nu_1, \ldots, \nu_n$. As $D$ has SNC-support, we may write
     \begin{align*}
         D = \sum_{i=1}^{s} \delta_i D_i
     \end{align*}
     in some analytic open neighbourhood of $y\in X$ so that the $D_i$ are smooth divisors meeting transversely at $y$ and where $s\leq n$. Then the $\mathcal{N}^*_{D_i/X}|_{y}\subseteq \Omega_X^1|_{y}$ are distinct, one-dimensional, $T$-invariant subspaces. In particular, by Proposition \ref{ExampleComputationLocalWeights} and after possibly re-indexing the $D_i$, we may assume that for any $i=1, \ldots, s$ the weight of the action of $T$ on $\mathcal{N}^*_{D_i/X}|_{y} \subseteq \Omega_X^1|_{y}$ is given by $\nu_i$. Note that $\delta_i\leq 1$ as $(X, D)$ is assumed to be sub-log canonical. Let us set $\delta_i = 0$ for $s+1 \leq i \leq n$. Write $A = \sum a_j A_j$, where, by assumption, $a_j \geq 0$ for all $j$.

     Then by Proposition \ref{ExampleComputationLocalWeights} the weight $\mu$ of the action of $T$ on $\mathcal{O}_X(-(K_X+ D - A))|_y$ satisfies the following equation:
     \begin{align}
         \mu = - \sum_{i=1}^{n} \nu_i + \sum_{i=1}^{n} \delta_i \nu_i - \sum_j a_j m_{j, i} \ \nu_{i} 
         = - \sum_{i=1}^{n} (1-\delta_i) \nu_i  - \sum_{i, j} a_j m_{j,i} \cdot \nu_{i}
         \in M_{\mathds{R}},\label{eqs:EquivNonVan AmpleCase 1b}
     \end{align}
     for some integers $m_{i, j} \geq 0$.

     On the other hand, it is an important fact that if $\mathscr{L}$ is any $T$-linearised ample line bundle on $X$ anf if $\mu$ denotes the weight of the action of $T$ on $\mathscr{L}|_y$ then 
     \begin{align}
         \mu + \varepsilon \nu_i \in P_{\mu}(X, \mathscr{L})\label{eqs:EquivNonVan AmpleCase 1c}
     \end{align}
     for sufficiently small $\varepsilon \geq 0$, see \cite[Corollary 2.14]{wisniewski_AlgebraicTorusActionsOnContactManifolds} for a quick and purely algebraic proof \footnote{In fact, the original statement in \cite[Theorem 3]{GuilleminSternberg_ConvexityPropertiesOfMomentMaps} holds much more generally for any symplectic manifold and even shows that locally around $\mu$, $P_{\mu}(X, \mathscr{L})$ is just the cone spanned by the $\nu_1, \ldots, \nu_n$, see Figure \ref{fig:MomentPolytope} below. Note that unfortunately \cite{GuilleminSternberg_ConvexityPropertiesOfMomentMaps} conventions regarding moment maps differ from ours by a minus sign.}.
     \begin{figure}[h]
         \centering
        \includegraphics[height= 4cm]{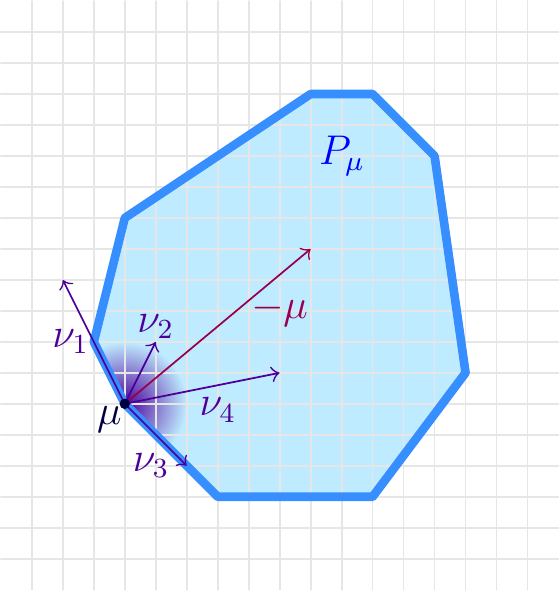}
        \caption{Locally around the vertex $\mu$ the moment polytope $P_\mu$ (drawn in blue) looks like the cone spanned by the weights of the action on the cotangent space (drawn in lilac).}\label{fig:MomentPolytope}
    \end{figure}

    \noindent
    In any case in view of (\ref{eqs:EquivNonVan AmpleCase 1b}) and (\ref{eqs:EquivNonVan AmpleCase 1c}) we see that
    \begin{align}
        (1-\varepsilon) \mu = \mu + \varepsilon \left( \sum_{i=1}^{n} (1-\delta_i) \nu_i  + \sum_{i, j} a_j m_{j,i}\cdot \nu_{i} \right) \in P_{\mu}\Big(X, \mathcal{O}_X\big({-}(K_X+ D  - A)\big)\Big),\label{eqs:EquivNonVan AmpleCase 2}
    \end{align}
    where we use that $\delta_i\leq 1$ as $(X, D)$ is sub-log canonical and $a_j \geq 0$ as $A$ is effective. This proves (\ref{eqs:EquivNonVan AmpleCase 1}) as required. Finally, as was indicated already above, the claimed Lemma \ref{EquivariantNonVanishing Ample Case} now follows immediately from (\ref{eqs:EquivNonVan AmpleCase 1}) and the following elementary consideration:
    \end{proof}

     \begin{proposition}
        Let $V$ be a finite dimensional real vector space, let $W\subseteq V$ be a finite subset and let $P\subseteq V$ be its convex hull. Assume that for any $w\in W$ it holds that $(1-\varepsilon)w \in P$ for all sufficiently small $0\leq \varepsilon$. Then $0\in P$.\label{Convex Geometry}
    \end{proposition}
    \begin{proof}
        Let us assume that $0\notin P$. Then there exists a linear functional $\varphi$ on $V$ such that $\varphi|_P > 0$. As $P\subseteq V$ is compact and convex so is $\varphi(P)\subseteq \R$, i.e. it is a (compact) interval, say $\varphi(P) = [a,b]$. Note that as $P$ is the convex hull of $W$, $\varphi(P)$ is the convex hull of $\varphi(W)$ and so there exists at least one $w\in W$ for which $\varphi(w) = a$. But then our assumption that $(1-\varepsilon)w \in P$ for all sufficiently small $0<\varepsilon$ contradicts the fact that 
        \begin{align*}
            \varphi\Big((1-\varepsilon)w\Big) = (1-\varepsilon) \cdot \varphi(w) = (1-\varepsilon) \cdot a \notin [a,b].
        \end{align*}
        This completes the proof of the proposition.
    \end{proof}

    \begin{corollary}
        Let $(X, D)$ be a projective, log smooth, log canonical pair and fix an algebraic torus $T\subseteq \Aut(X, D)$. If $-(K_X+ D)$ is nef, then $0 \in P_{\mu}(X, -(K_X+ D))$.\label{Equivariant NonVanishing: Semiample case almost}
    \end{corollary}
    \begin{proof}
    Let us start by proving the following:\vspace{0.5\baselineskip}
    
    \emph{Claim:} There exists a $T$-invariant, effective, ample divisor $A$ on $X$.\vspace{0.5\baselineskip}
       
        Indeed, let $\mathscr{L}$ be a very ample line bundle on $X$. Then $\mathscr{L}$ is $T$-invariant and, possibly replacing $\mathscr{L}$ by some multiple, even linearisable, see \cite[Theorem 5.2.1]{brion_LinearisationsOfGroupActions}. Let us choose an eigenvector $\sigma \in \mathrm{H}^0(X, \mathscr{L})$ for the action of $T$ of weight $w \in M$. Then the zero-set $A := Z(\sigma)$ is a $T$-invariant, effective, ample divisor on $X$ as required.

        Now, write $A = \sum_j a_j A_j$ as the sum of distinct prime divisors. For any rational number $\varepsilon > 0$ we consider the ample (!) divisor 
        \begin{align*}
            -(K_X + D - \varepsilon A) = -(K_X+ D) + \varepsilon A.
        \end{align*}
        According to Lemma \ref{EquivariantNonVanishing Ample Case} it holds that
        \begin{align*}
            0 \in P_{\varepsilon} := P_{\mu}(X,-(K_X+ D - \varepsilon A)), \quad \varepsilon >0.
        \end{align*}
        We are now ready to conclude that $0 \in P := P_0 := P_{\mu}(X, -(K_X+ D))$: Let us denote by
        \begin{align*}
            d(P_\varepsilon, P) := \max\Big\{\hspace{0.1cm} \sup_{w \in P } d(P_\varepsilon, w), \sup_{v\in P_\varepsilon}d(P, v) \hspace{0.1cm} \Big\} 
        \end{align*}
        the Hausdorff distance between $P_\varepsilon$ and $P$ with respect to some norm on $M_{\mathds{R}}$. Note that by definition and for any $\varepsilon \geq 0$, the polytope $P_{\varepsilon}$ is the convex hull of the weights $\mu_{\varepsilon, k}$ of the action of $T$ on $ \mathcal{O}_X(-(K_X+ D - \varepsilon A))|_{y_k}$. Moreover, by (\ref{eqs:EquivNonVan AmpleCase 1b}) we have
        \begin{align*}
            \mu_{0, k} = \mu_{\varepsilon, k} + \varepsilon \left( \sum_{i, j} a_{j} m^{(k)}_{j, i} \cdot \nu_{i} \right)
        \end{align*}
        for some integers $m^{(k)}_{j, i} \geq 0$, depending only on the weights of the action of $T$ on $\mathcal{O}_X(-A_j)|_{y_k}$ but not on $\varepsilon \geq 0$, c.f.\ Proposition \ref{ExampleComputationLocalWeights}. We infer that clearly $d(P_\varepsilon, P_0)\rightarrow 0$ as $\varepsilon \rightarrow 0$. In particular, $d(P_0, 0) \leq d(P_\varepsilon, P_0)\rightarrow 0$ and so $0\in P = P_0$, using that the latter set is closed.
    \end{proof}

    \begin{proof}[Proof (of Theorem \ref{Equivariant Non-Vanishing: Semi ample case})]
        Follows immediately by simply combining  Corollary \ref{Equivariant NonVanishing: Semiample case almost} with Proposition \ref{PropertiesSectionMomentPolytope}.
    \end{proof}

    \noindent
    The rest of this section is devoted to giving examples which demonstrate that the assumptions in Theorem \ref{Equivariant Non-Vanishing: Semi ample case}, and hence in Theorem \ref{Intro:Equivariant Non-Vanishing}, are essentially optimal.
    \begin{example}
        The conclusion of Theorem \ref{Equivariant Non-Vanishing: Semi ample case} fails if we allow for more general subgroups $H \subseteq \Aut(F, \Delta_F)$. In fact, the following example shows that it is already wrong for solvable groups:

        Let $X = \CP^1$, $D = 0$ and let $\C^\times \cong T\subseteq \textmd{PGL}_2(\C)$ be the group of diagonal matrices as before. Then one easily checks that 
        \begin{align*}
            \bigoplus_{m= 0}^\infty H^0(X, \mathcal{O}_X(-mK_X))^T = \C[z \partial_z]
        \end{align*}
        is generated by a single section. However, the vector field $z \partial_z$ is \emph{not} invariant under the action of the subgroup $B\subseteq \textmd{PGL}_2(\C)$ of upper triangular matrices.
        \label{ex:EquivKodDim-Infty}
    \end{example}
    \begin{example}
        The above example shows that usually $\kappa(X, -(K_X + D))^T$ is strictly smaller than $\kappa(X, -(K_X + D))$ even when the assumptions in Theorem \ref{Equivariant Non-Vanishing: Semi ample case} are satisfied. In fact, it is straightforward to extend Example \ref{ex:EquivKodDim-Infty} to show that
        \begin{align*}
            0 = \kappa\big(\CP^n, -K_{\CP^n}\big)^T < \kappa\big(\CP^n, -K_{\CP^n}\big) = n,
        \end{align*}
        for any $n\geq 1$. Here, $T\subseteq \Aut(\CP^n) = \textmd{PGL}_{n+1}(\C)$ denotes the torus of diagonal matrices.\label{ex:InvariantKodDimSmallerThanKodDim}
    \end{example}
    \begin{example}
        The assumption that $(X, D)$ is sub-log canonical is optimal as the following example demonstrates: Consider $X = \CP^1$ equipped with $D = a\cdot \{0\}$ for some $a\in \Q$. Then the standard torus of diagonal matrices $T \cong \C^\times \subseteq \Aut(\CP^1)$ acts on $\CP^1$ via $t\bigcdot [x:y] = [tx:y]$. Note that the weight of the action of $T$ on $\mathcal{T}_{\CP^1}|_0$ is given by $1\in \Z\cong M$ and the weight of the action of $T$ on $\mathcal{T}_{\CP^1}|_\infty$ is given by $-1\in M$. Consequently, the weights of the action of $T$ on $\Omega^1_{\CP^1}|_0$ and $\Omega^1_{\CP^1}|_\infty$ are given by $-1$ and $1$ respectively. Thus, $P_{\mu}(\CP^1, \mathcal{O}(-K_{\CP^1})) = [-1, 1]$ and, hence, using Example \ref{ExampleComputationLocalWeights}
        \begin{align*}
            P_{\mu} \Big( \CP^1, -\big( K_{CP^1} + D \big) \Big) = \Big[-1+a, \hspace{0.1cm} 1\Big].
        \end{align*}
         We conclude that $-(K_{\CP^1}+D)$ has $T$-invariant sections if and only if $a\leq 1$ which is the case if and only if $(X, D)$ is sub-log canonical. At the same time, notice that $-(K_{\CP^1}+D)$ is ample as long as $a\leq 2$.\label{ex:LogCanonicityIsOptimal}
    \end{example}


 \section{Proof of the main results}
 \label{sec:Proof Of Main Results}

 \subsection{Proof of Theorem C}

    Let $(X, D)$ be a projective, sub-log canonical pair and assume that $-(K_X + D)$ is semiample. Fix a connected, commutative, linear algebraic subgroup $H\subseteq \Aut(X, D)$. We need to show that
    \begin{align*}
        \kappa(X, -(K_X + D))^H\geq 0.
    \end{align*}
    If $(X, D)$ is log smooth, then this was already shown in Theorem \ref{Equivariant Non-Vanishing: Semi ample case}, c.f.\ also Lemma \ref{InvariantsCommutativeGroups}. In general, we choose an $H$-equivariant log resolution $f\colon \hat{X} \rightarrow X$ of $(X, D)$. Such a map always exists by  \cite[Proposition 3.9.1]{kollar_ResolutionOfSingularities}. Then, in particular $\textmd{Exc}(f)$ will be $H$-invariant.

    Let us write
     \begin{align}
         K_{\hat{X}} + \hat{D} \sim_{\mathds{Q}} f^*(K_X + D), \label{eqs:EquivNonVan Singular Semiample Case}
     \end{align}
     where $\hat{D}$ is supported on the strict transform of $D$ and the exceptional locus of $f$. Then $(\hat{X}, \hat{D})$ is a log smooth pair which is sub-log canonical as $(X, D)$ is so, see \cite[Lemma 2.30]{kollar_BirationalGeometry}. Also, by construction, 
     \begin{align*}
        -\big(K_{\hat{X}} + \hat{D} \big) \sim_{\mathds{Q}} -f^*(K_X + D),
     \end{align*}
     is semiample. Thus the conditions in Theorem \ref{Equivariant Non-Vanishing: Semi ample case} are satisfied and we deduce that
     \begin{align*}
         \kappa\left(\hat{X}, {-} \left(K_{\hat{X}} + \hat{D} \right)\right)^H \geq 0.
     \end{align*}
     Let us fix an integer $m\geq 1$ such that $-m(K_{\hat{X}} + \hat{D} )$ is Cartier and such that there exists an $H$-invariant form
     \begin{align*}
         \hat{\tau} \in \hspace{0.1cm }& H^0\Big(\widehat{X}, \mathcal{O}_{\widehat{X}}\Big({-}m\big(K_{\hat{X}} + \hat{D} \big)\Big)\Big)\\
         \cong \hspace{0.1cm }& H^0\Big(\widehat{X}, \mathcal{O}_{\widehat{X}}\big(f^*({-}m(K_X + D))\big)\Big) \\
         = \hspace{0.1cm }& H^0\big(X, \mathcal{O}_{X}\big({-}m(K_X + D)\big)\big).
     \end{align*}
     Here, for the isomorphism in the second line we used (\ref{eqs:EquivNonVan Singular Semiample Case}). Let us denote by $\tau \in H^0(X, \mathcal{O}_{X}(-m(K_X + D)))$ the image of $\hat{\tau}$ under this isomorphism. We claim that $\tau$ is $H$-invariant, concluding the proof of the Lemma. Note that this is not entirely clear as the natural $H$-actions on both sides may a priori differ.

     However, as $f\colon \hat{X} \rightarrow X$ is birational and $H$-equivariant, the set $U := \hat{X} \setminus (\hat{D} \cup \textmd{Exc}(f))$ is a dense open $H$-invariant subset of $\hat{X}$ on which $f$ is an isomorphism. In particular, we see that
     \begin{align*}
         \hat{\tau}|_U = \tau|_U \in H^0(U, \mathcal{O}_X(-K_U)).
     \end{align*}
     As $\hat{\tau}$ is $H$-invariant we infer that so is $\hat{\tau}|_U = \tau|_U$ and, hence, (by continuity) $\tau$. Thus, we have produced a non-trivial $H$-invariant form $\tau \in H^0(X, \mathcal{O}_{X}(-m(K_X + D)))$ and we are done.\hfill $\square$

 \subsection{Proof of Theorem A}

 Theorem \ref{Intro:Anti-Non-Vanishing} is an immediate consequence of Theorem \ref{Intro:Equivariant Non-Vanishing} and Theorem \ref{criterionNonVanishing}.\hfill $\square$

 \subsection{Proof of Corollary B}
Let $(X, \Delta)$ a projective klt threefold pair with nef anti-log canonical divisor $-(K_X+\Delta)$. We want to prove that
\begin{align}
    \kappa(X, -(K_X + \Delta))\geq 0\label{eqs:NonVanishing}.
\end{align}
Let us start with some preliminary considerations. Let us denote by $(F, \Delta_F)$ a general fibre of the MRC fibration of $(X, \Delta)$ as in Theorem \ref{StructureVarietiesNefACBundle}. Then $F$ is a rationally connnected projective variety, $(F, \Delta_F)$ is a klt pair and $-(K_F + \Delta_F)$ is nef. We distinguish four cases according to whether $\dim F = 0,1,2,3$.

If $\dim F = 3$, then $X=F$ and (\ref{eqs:NonVanishing}) follows from \cite[Theorem A]{lazic_NumericalNonVanishing}. In the other cases, according to Theorem \ref{criterionNonVanishing} we need to prove that for any algebraic torus $T\subseteq \Aut(F, \Delta_F)$ it holds that
\begin{align}
    \kappa(F, -(K_F + \Delta_F))^T\geq 0\label{eqs:NonVanishing2}.
\end{align}
In case $\dim F = 0$, i.e.\ if $F = \{ * \}$ is a point, there is nothing to prove, c.f.\ Theorem \ref{StructureVarietiesNefACBundle}. Also the case $\dim F = 1$ is easy to settle, for example because in this case clearly $F= \CP^1$. Then since $-(K_F + \Delta_F)$ is nef it is even semiample and (\ref{eqs:NonVanishing}) is a direct consequence of Theorem \ref{Intro:Anti-Non-Vanishing}. Thus, it remains to deal with the case when $F$ is a surface.

So let us then assume that $F$ is a surface. If $T = \{1\}$ is trivial, then we simply need to verify that
\begin{align*}
    \kappa(F, -(K_F + \Delta_F))\geq 0,
\end{align*}
which was proved in \cite[Theorem 1.5]{han_NumericalNonvanishingSurfacPairs}. Alternatively, this statement is also contained in \cite[Theorem A]{lazic_NumericalNonVanishing}. In any case, note that this is precisely the situation when $X = Y \times F$ is a product.

It remains to deal with the case that $\dim T \geq 1$. Proceeding exactly as in the proof of Theorem \ref{Intro:Equivariant Non-Vanishing} we may replace $F$ by a log resolution. Note that this may force us to allow for $(F, \Delta_F)$ to be only \emph{sub}-klt but this will not be an issue henceforth. In any case, we may assume $F$ to be smooth.

Now, as $\dim T \geq 1$, the surface $F$ is certainly a \emph{$T$-variety of complexity one}, i.e.\ a variety equipped with the action of an algebraic torus such that the maximal dimensional orbits have codimension one in $F$. But then $F$ is an \emph{Mori Dream Space} in the sense of \cite{Hu_MoriDreamSpaces}, so that $-(K_F + \Delta_F)$ is not only nef but also semiample. Hence, (\ref{eqs:NonVanishing}) follows from Theorem \ref{Intro:Equivariant Non-Vanishing}. 

The assertion that $F$ is a Mori Dream Space under the above hypothesis seems to be well-known, for the convenience of the reader we nevertheless provide a detailed explanation below: As $F$ is a smooth, rationally connected surface, it is in particular rational, i.e.\ birational to $\mathds{P}^2$. It follows that the class group $\mathrm{Cl}(F)$ is finitely generated. Thus, to prove that $F$ is a Mori Dream Space it only remains to show that the Cox ring of $F$ is finitely generated, see \cite[Proposition 2.9]{Hu_MoriDreamSpaces}. But indeed, Hausen and Süß in \cite[Theorem 1.3]{hausen_CoxRingVarietieswithTorusAction} can even determine an explicit finite list of generators and relations in terms of combinatorial data associated to the $T$-action. The theorem is proved.\hfill $\square$

\subsection*{Acknowledgment}
First and foremost, the author would like to express his sincere gratitude towards Jarosław Wiśniewski for sketching to the author the proof of Theorem \ref{Intro:Equivariant Non-Vanishing} in the case of a smooth Fanos without boundary. This paper would not exist without him. The author would also like to take this opportunity to warmly thank Roland Púček and Andriy Regeta for organising the `Regional Workshop in Algebraic Geometry' at Jena where the author and Jarosław Wiśniewski met.

Second, the author is extremely grateful towards the authors of \cite{lazic_NumericalNonVanishing} for bringing up the question whether Theorem \ref{Intro:Anti-Non-Vanishing} is true and, in particular, Shin-Ichi Matsumura and Thomas Peternell for several inspiring and fruitful initial discussions on the issue. Moreover, he is indebted to Vladimir Lazi\'c and to Nikolaos Tsakanikas for many helpful comments on an early draft of this text and especially to Nikolaos Tsakanikas and Lingyao Xie for pointing out a serious mistake in an earlier version of this text. 

He would also like to thank Daniel Greb for his advice and for always trying to make the author see the positive side of things, and Jochen Heinloth for several helpful discussions on torus-actions and for always being glad to answer any question of anyone stumbling by his office. Finally, he would like to warmly thank the anonymous referee for their detailed review and their valuable suggestions.

The author gladly wants to acknowledge that while writing this article he was supported from the DFG Research Training Group 2553 `\emph{Symmetries and Classifying Spaces: Analytic, Arithmetic and Derived}'.

\def\refname{References}

\end{document}